\documentclass[a4paper]{amsart}
\usepackage[T1]{fontenc}
 \usepackage[francais, english]{babel}
 \usepackage{amssymb,url,xspace}
 \usepackage{graphicx}
  \usepackage[all]{xy}
\theoremstyle{plain}

\usepackage[all]{xy}

\newtheorem{thm}{Theorem}

\newtheorem{cor}[thm]{Corollary}

\newtheorem{lemma}[thm]{Lemma}

\theoremstyle{definition}

\newtheorem{remark}[thm]{Remark}

\DeclareMathOperator{\SL}{SL}

\def\C{\mathbb C}

\def\sl{\mathfrak{sl}_2}

\begin{document}
%\frontmatter

\author[K. Diarra]{Karamoko DIARRA}
\address{ DER de Math\'ematiques et d'informatique, FAST, Universit\'e des Sciences, des Techniques et des Technologies de Bamako, BP: E $3206$ Mali.}
\email{diarak2005@yahoo.fr}
 \email{karamoko.diarra@univ-rennes1.fr}
 
\author[F. Loray]{Frank LORAY}
\address{ IRMAR - UMR 6625 du CNRS, Campus de Beaulieu, Universit\'e de Rennes 1, 35042 Rennes Cedex, France.}
\email{frank.loray@univ-rennes1.fr}

\title[Tame isomonodromic solutions]{Ramified covers and tame isomonodromic solutions on curves}

\begin{abstract}In this paper, we investigate the possibility of constructing isomonodromic deformations by ramified covers.
We give new examples and prove a classification result.
\end{abstract}
%\begin{altabstract}

%\end{altabstract}
%\altkeywords{\'Equations diff\'erentielles ordinaires, D\'eformations isomonodromiques, Familles de Hurwitz}
%\translator{hPrénom Nomi}
%\thanks{hSubventionsi}
\dedicatory{To Yulij Ilyashenko for his $70^{th}$ birthday}
\maketitle
%\tableofcontents 
%\mainmatterthm
%\begin{altabstract}
%\end{abstract}
\section*{Introduction}

Let $X$ be a complete curve of genus $g$ over $\mathbb C$ and $D$ be a reduced divisor on $X$:
$D=[x_1]+\cdots+[x_n]$ is equivalent to the data of $n$ distinct points on $X$. Set $N:=3g-3+n$;
when $N>0$, that we will assume along the paper, then $N$ is the dimension of the deformation space 
$M_{g,n}$ of the pair $(X,D)$. 

Let $(E,\nabla)$ be a rank $2$ logarithmic connection over $X$ with polar divisor $D$. In other words, 
$E\to X$ is a rank $2$ holomorphic vector bundle and $\nabla:E\to E\otimes\Omega^1_X(D)$
a linear meromorphic connection having simple poles at the points of $D$. By considering the analytic
continuation of a local basis of $\nabla$-horizontal  sections of $E$, we inherit a monodromy representation
$$\rho_\nabla\ :\ \pi_1(X\setminus D)\to\mathrm{GL}_2(\mathbb C)$$
(which is well-defined up to conjugacy in $\mathrm{GL}_2(\mathbb C)$). 

Given a deformation $t\mapsto (X_t,D_t)$ of the complex structure, there is a unique deformation 
$t\mapsto (X_t,D_t,E_t,\nabla_t)$ up to bundle isomorphism such that the monodromy is constant. 
For $t$ varrying in the Teichmuller space $T_{g,n}$, we get the universal isomonodromic deformation (see \cite{Heu}). 
Considering the moduli space $\mathcal M_{g,n}$ of quadruples $(X,D,E,\nabla)$,
isomonodromic deformations define the leaves of a $N$-dimensional foliation transversal to the natural projection 
$$\mathcal M_{g,n}\to M_{g,n}\ ;\ (X,D,E,\nabla)\mapsto (X,D).$$
The corresponding differential equation is explicitely described in \cite{Krichever} (via local analytic coordinates on $M_{g,n}$)
and is known to be polynomial with respect to the algebraic structure of $\mathcal M_{g,n}$ 
(it is the non-linear Gauss-Manin connection constructed in \cite[section 8]{Simpson}).
In the case $g=0$, we get the Garnier system (see \cite{OkamotoGarnier}), and for $n=4$, the Painlev\'e VI equation. 
Solutions (or leaves) are generically transcendental and it is expected that the transcendence increase with $N$
(see \cite[Introduction]{DubrovinMazzocco} for instance). However, there are some tame solutions. 

Algebraic solutions of Painlev\'e VI equation were recently classified in \cite{Boalch,LisovyyTykhyy}.
Some algebraic solutions are constructed in \cite{DiarraBoletim} for the Garnier case; 
see the discussion in the introduction of \cite{DiarraToulouse} for higher genus case.

Some solutions, called ``classical'', reduce to solutions of linear differential equations. 
They are classified in the Painlev\'e case in \cite{Watanabe}. In the Garnier case,
such solutions arise by considering deformations of reducible connections (see \cite{OkamotoKimura,MazzoccoGarnier}):
they can be expressed in terms of Lauricella hypergeometric functions.

There are also ``tame solutions'' coming from simpler isomonodromy equations (e.g. with lower $n$ or $g$) . 
In \cite{MazzoccoGarnier}, it is proved that, one way of reducing $n$ (when $g=0$) is to consider those deformations having scalar 
local monodromy around some pole. There is another way of reduction, by using ramified covers, 
and this is what we want to investigate in this note.

\section{Known constructions via ramified covers}\label{SecKnownConstructions}

Ramified covers of curves have already been used to construct algebraic solutions of the Painlev\'e VI equation 
(see \cite{Doran,AndreevKitaev}) and Garnier systems (see \cite{DiarraBoletim}).
But they have also been used to understand relations between transcendental solutions. 

\subsection{}\label{SecQuadratic}The most classical 
case is the {\bf quadratic transformation} of the Painlev\'e VI equation 
(see \cite{KitaevQuadratic,Manin,TsudaOkamotoSakai,MazzoccoVidunas}).
We consider a deformation $t\mapsto(E_t,\nabla_t)$ of a rank $2$ connection on $\mathbb P^1$
with simple poles at $(x_1,x_2,x_3,x_4)=(0,1,t,\infty)$. At a pole $x_i$, we consider {\it eigenvalues} $\theta_i^1,\theta_i^2$ of the residual matrix
and call {\it exponent} the difference $\theta_i:=\theta_i^1-\theta_i^2$ (defined up to a sign).
To be concrete, if all $\theta_i^1+\theta_i^2=0$ and the connection is irreducible, then $E_t$ is the trivial bundle except for 
a discrete set of parameters (see \cite{Bolibruch}) and the connection is just defined by a two-by-two system.
If moreover exponents satisfy
$$\theta_0=\theta_\infty=\frac{1}{2}$$
then after lifting the connection on the two-fold cover 
$$\mathbb P_{\tilde x}^1\to\mathbb P_{x}^1\ ;\ \tilde x\mapsto \tilde x^2$$
we get a connection $(\tilde E_t^0,\tilde\nabla_t^0)$ having $6$ simple poles at 
$$\tilde x=0,\ \pm1,\ \pm\sqrt{t}\ \text{and}\ \infty$$
(see figure \ref{dessin1}).

\begin{figure}[htbp]
\begin{center}
\includegraphics[scale=0.8]{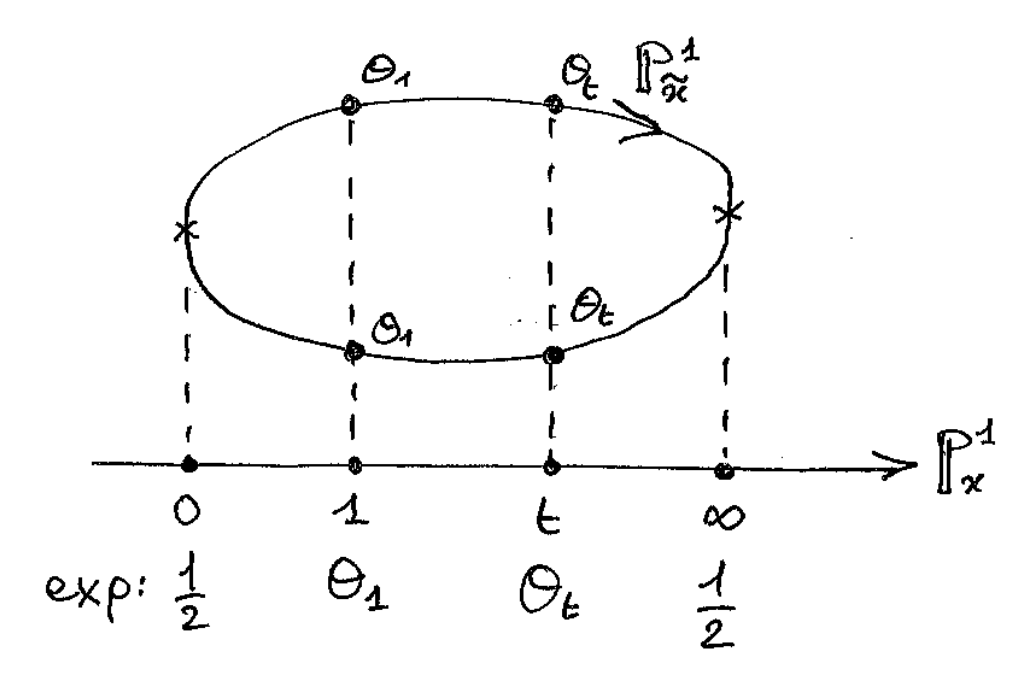}
\caption{{\bf Quadratic transformation's cover}}
\label{dessin1}
\end{center}
\end{figure}

Those two poles at ramification points $\tilde x=0,\infty$ have now integral exponents and therefore scalar local monodromy $-I$.
These singular points are ``apparent'', i.e. can be erased by a combination of
\begin{itemize}
\item a rational gauge (i.e. birational bundle) transformation,
\item the twist by a rank $1$ connection.
\end{itemize}
This can be done 
taking into account the deformation, and we get a new deformation $t\mapsto (\tilde E_t,\tilde\nabla_t)$ 
of a rank $2$ connection with $4$ simple poles $\tilde x=\pm1$ and $\pm\sqrt{t}$ on the Riemann sphere $\mathbb P_{\tilde x}^1$.
This new deformation is clearly isomonodromic if the initial deformation was.
Taking into account the exponents, we get a {\it rational two-fold cover
$$\mathrm{Quad}\ :\ \mathcal M_{0,4}(\frac{1}{2},\theta_1,\theta_t,\frac{1}{2})\ \stackrel{2:1}{\longrightarrow}\ \mathcal M_{0,4}(\theta_1,\theta_1,\theta_t,\theta_t)$$
between moduli spaces that conjugates isomonodromic foliations}. The map $\mathrm{Quad}$ is called quadratic transformation 
of the Painlev\'e VI equation.

\subsection{}\label{SecQuartic}When exponents satisfy $\theta_0=\theta_1=\theta_\infty=\frac{1}{2}$, we can iterate twice the map
(after conveniently permuting the poles) and we get the {\bf quartic transformation}
$$\mathrm{Quad}\circ\mathrm{Quad}\ :\ \mathcal M_{0,4}(\frac{1}{2},\frac{1}{2},\theta_t,\frac{1}{2})\ \stackrel{4:1}{\longrightarrow}\ \mathcal M_{0,4}(\theta_t,\theta_t,\theta_t,\theta_t).$$
Finally, if we consider the Picard parameters $\theta_0=\theta_1=\theta_t=\theta_\infty=\frac{1}{2}$ for Painlev\'e VI equation,
we can iterate arbitrary many times the quadratic tranformation. There is also a cubic transformation in this case
(see \cite{MazzoccoVidunas}). % and many others coming from group law on elliptic curves (see below).

\subsection{}\label{SecPicard}For {\bf Picard parameters} $(\theta_0,\theta_1,\theta_t,\theta_\infty)=(\frac{1}{2},\frac{1}{2},\frac{1}{2},\frac{1}{2})$ of Painlev\'e VI equation, 
one can modify the construction above as follows. Consider now the elliptic two-fold cover ramifying over the $4$ poles of $(E_t,\nabla_t)$
$$\phi_t\ :\ X_t=\{y^2=x(x-1)(x-t)\}\stackrel{2:1}{\rightarrow}\mathbb P^1_x\ ;\ (x,y)\mapsto x$$
and lift-up the connection on the elliptic curve. After birational gauge transformation, we get a holomorphic connection $(\tilde E_t,\tilde \nabla_t)$
that generically split as the direct sum of two holomorphic connections of rank $1$. This means that, for these parameters, Painlev\'e VI
solutions actually parametrize isomonodromic deformations of rank $1$ connections over a family of elliptic curves. This allow to solve
this very special element of Painlev\'e VI family (originally found by Picard) by means of elliptic functions (see \cite{Hitchin,MazzoccoPicard,Lame0}). By the way, we get {\it a birational map
$$\mathcal M_{0,4}(\frac{1}{2},\frac{1}{2},\frac{1}{2},\frac{1}{2})\ \stackrel{\sim}{\longrightarrow}\ \mathcal M_{1,0}$$
that commutes with isomonodromic flow}. 
%Now, if we denote by $\phi_n:X_{t}\to X_{\tilde t}$ the multiplication by $n$ ... induces a generically finite map
%$$\Phi\ :\ \mathcal M_{1,0}\ \stackrel{n:1}{\longrightarrow}\ \mathcal M_{1,0};$$
%conjugating by the previous birational map, we get a generically finite map
%$$\Psi^{-1}\circ\Phi\circ\Psi\ :\ \mathcal M_{0,4}(\frac{1}{2},\frac{1}{2},\frac{1}{2},\frac{1}{2})\ \stackrel{n:1}{\longrightarrow}\ \mathcal M_{0,4}(\frac{1}{2},\frac{1}{2},\frac{1}{2},\frac{1}{2}).$$

\begin{figure}[htbp]
\begin{center}
\includegraphics[scale=0.8]{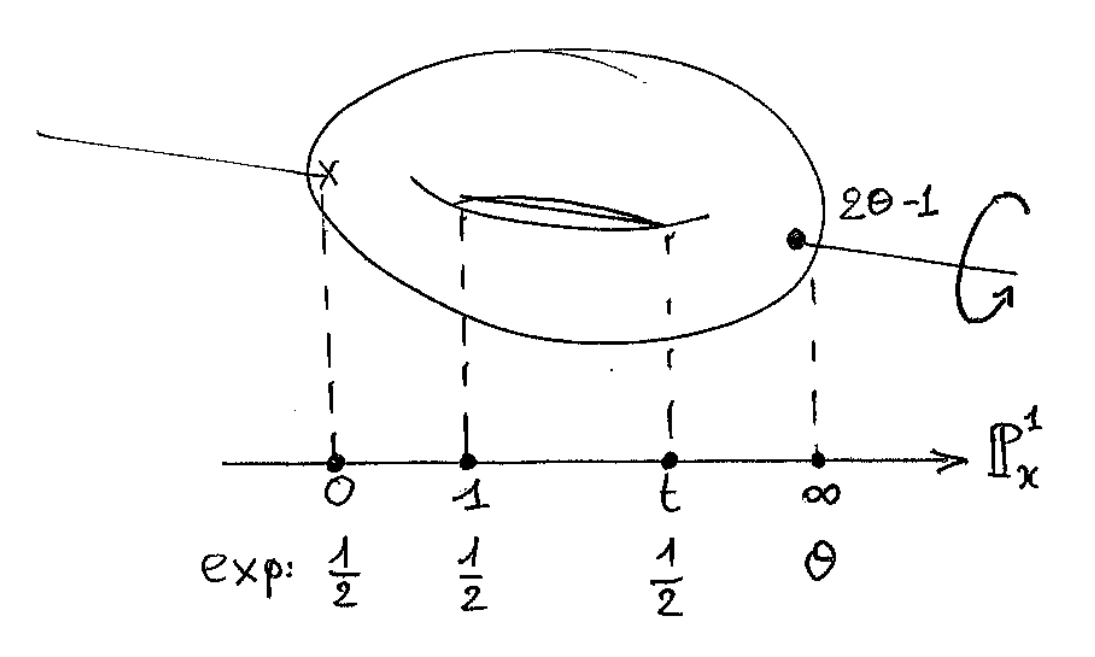}
\caption{{\bf Lam\'e's cover}}
\label{dessin2}
\end{center}
\end{figure}

\subsection{}\label{SecLame}This map has been extended to {\bf Lam\'e parameters} in \cite{Lame1,Lame2} as a birational transformation
$$\text{Lam\'e}\ :\ \mathcal M_{0,4}(\frac{1}{2},\frac{1}{2},\frac{1}{2},\theta_\infty)\ \stackrel{\sim}{\longrightarrow}\ \mathcal M_{1,1}(2\theta_\infty-1)$$
also commuting with isomonodromic flow (see figure \ref{dessin2}).

\subsection{}\label{SecGenus2}In \cite{HeuLoray}, a $2$-fold ramified cover commuting with isomonodromic flow
$$\mathcal M_{0,6}(\frac{1}{2},\frac{1}{2},\frac{1}{2},\frac{1}{2},\frac{1}{2},\frac{1}{2})\ \stackrel{2:1}{\longrightarrow}\ \mathcal M_{2,0}$$
has been constructed by lifting connections on the hyperelliptic cover 
$$\phi_{r,s,t}\ :\ X_{r,s,t}=\{y^2=x(x-1)(x-r)(x-s)(x-t)\}\ \stackrel{2:1}{\longrightarrow}\ \mathbb P^1_x\ ;\ (x,y)\mapsto x$$
(see figure \ref{dessin3}).

\begin{figure}[htbp]
\begin{center}
\includegraphics[scale=0.8]{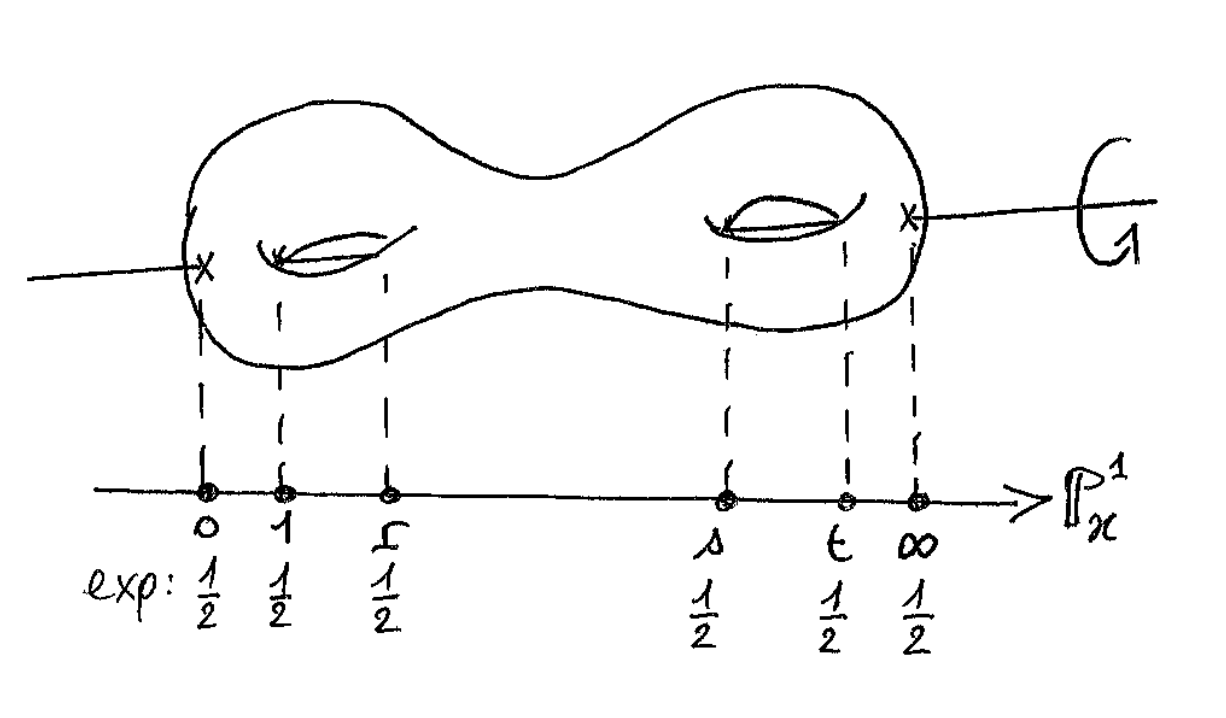}
\caption{{\bf Genus $2$ cover}}
\label{dessin3}
\end{center}
\end{figure}

\subsection{}However, for higher genus $g>2$ hyperelliptic curve, the similar map 
$$\mathcal M_{0,2g+2}(\frac{1}{2},\ldots,\frac{1}{2})\ \longrightarrow\ \mathcal M_{g,0}$$
has small image: not only the deformation upstairs is reduced to the hyperelliptic locus (having codimension $g-2$), 
but even for a fixed hyperelliptic curve, the image has codimension $2(g-1)$ in the moduli space of connections. 

\section{Results}\label{SecResults}

In this note, we classify all ``interesting'' maps that can be constructed between moduli spaces like above,
using ramified covers of curves. Let us explain.

Let $(X,D^\nabla,E,\nabla)$ be a logarithmic rank $2$ connections and $\phi:\tilde X\to X$ be a ramified cover.
Let $D^\phi$ denotes the set of critical points of $\phi$ while $D^\nabla$ denotes the set of poles of $\nabla$;
they will be not disjoint in many cases. Consider now the universal deformation $t\mapsto(X_t,D_t)$
of the marked curve $(X,D)$ where $D$ is the union of $D^\phi$ and $D^\nabla$.
There is a unique local deformation $t\mapsto(X_t,D_t,E_t,\nabla_t,\phi_t)$ where $t\mapsto(X_t,D^\nabla_t,E_t,\nabla_t)$
is isomonodromic, and $t\mapsto(X_t,D^\phi_t,\nabla_t)$ is topologically trivial (we just deform the critical locus $D^\phi_t$).
Fibers of the map $t\mapsto(X_t,D^\nabla_t)$ are algebraic deformations, so-called Hurwitz families.

The main remark is that the lift to $\tilde X_t$ of the connection:
$$t\mapsto(\tilde E_t,\tilde\nabla_t):=\phi_t^*(E_t,\nabla_t)$$
is isomonodromic along the deformation. By applying rational gauge transformation and twisting with a rank $1$ isomonodromic deformation,
we may assume that $(\tilde E_t,\tilde\nabla_t)$ is an isomonodromic deformation of logarithmic $\sl$-connexion, free of apparent singular points.
In fact, this is possible whenever $\nabla_t$ has an essential singular point, i.e. with monodromy. Let $\tilde D_t$ be the 
(reduced) polar divisor of $\tilde\nabla_t$ after deleting apparent singular points.
Last but not least, assume that 
\begin{itemize}
\item the connection $(E_t,\nabla_t)$, or equivalently $(\tilde E_t,\tilde\nabla_t)$, has Zariski dense monodromy,
\item the deformation $t\mapsto (\tilde X_t,\tilde D_t,\tilde E_t,\tilde\nabla_t)$ induces a locally universal deformation 
$t\mapsto (\tilde X_t,\tilde D_t)$
of the marked curve.
\end{itemize}
These are the so-called ``interesting'' conditions.
The second item means that we get a complete isomonodromic deformation after the construction. We thus get 
a complete parametrisation of a leaf of the isomonodromic foliation. All examples listed in section \ref{SecKnownConstructions}
are examples of such constructions. It is easy to construction many exemples where all conditions but the last one are satisfied.
However, the last condition, saying that we get the complete deformation, is so hard to realized that we are able, 
in section \ref{SecClassification}, to classify all examples. This is our main result in this note. Appart from above 
known examples, we have the following three new cases.

\subsection{}\label{SecUncompleteTwicePuncturedTorus}

Let $s\mapsto X_s=\{y^2=x(x-1)(x-s)\}$ the Legendre family of elliptic curves and let $t\mapsto(E_t,\nabla_t)$
an isomonodromic deformation of a rank $2$ connection with poles located at $x=0,1,t,\infty$. More rigorously, we should  say
$\tilde t\mapsto(E_t,\nabla_t)$ where $\tilde t$ belongs to the Teichmuller space, 
given by the universal cover $T\to\mathbb P^1_x\setminus\{0,1,\infty\}$ in this case, 
and $t$ denotes the projection of $\tilde t$ on $\mathbb P^1_x\setminus\{0,1,\infty\}$.
Now, assume that exponents of $\nabla_t$ take the form $(\theta_0,\theta_1,\theta_t,\theta_\infty)=(\frac{1}{2},\frac{1}{2},\theta,\frac{1}{2})$.
Therefore, after lifting on the elliptic curve, we get a connection with $3$ apparent singular points and two copies 
of the singular point at $x=t$. By gauge transformation, we finally get a connection $(\tilde E_{s,t},\tilde\nabla_{s,t})$
with only two simple poles, but to get a $\mathfrak{sl}_2$-connection we need to shift one of the two exponents (see figure \ref{dessin4}).
We finally get a rational map 
$$\mathbb P^1_s\times\mathcal M_{0,4}(\frac{1}{2},\frac{1}{2},\theta_t,\frac{1}{2})\ \longrightarrow\ \mathcal M_{1,2}(\theta,\theta-1).$$

\begin{figure}[htbp]
\begin{center}
\includegraphics[scale=0.8]{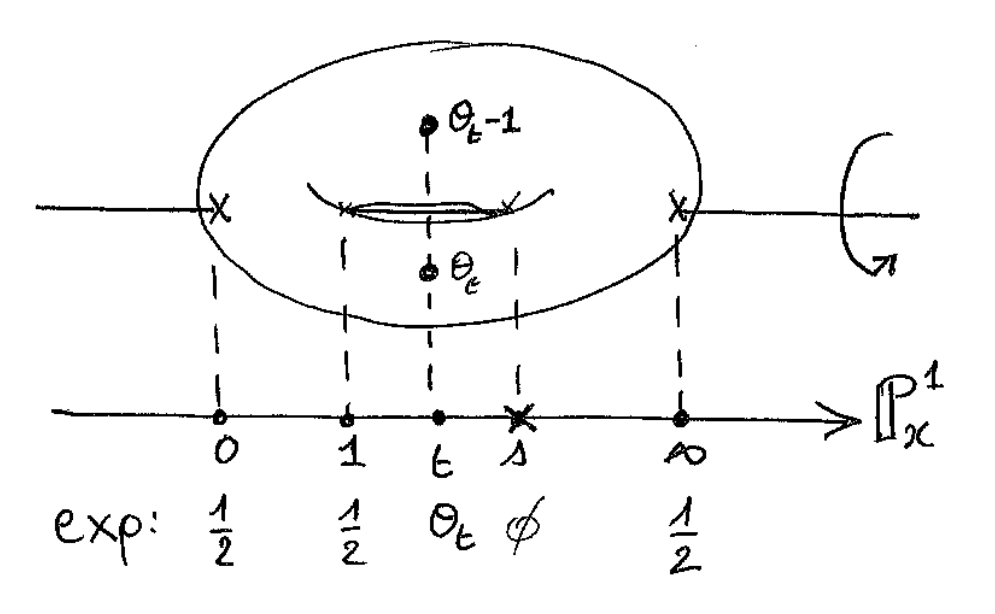}
\caption{{\bf Ruled deformations via uncomplete elliptic cover}}
\label{dessin4}
\end{center}
\end{figure}

Each isomonodromic deformation thus obtained is parametrized by a combination of a Painlev\'e VI solution (variable $t$) 
and a rational function (variable $s$). We get a $2$-parameter space of such tame isomonodromic deformations;
they form a codimension $2$ subset in $\mathcal M_{1,2}(\theta,\theta-1)$, the image of the map above, 
which is saturated by the isomonodromic foliation. The leaves belonging to this set are ruled surfaces parametrized by a Painlev\'e transcendent.
One should recover the Lam\'e case of section \ref{SecKnownConstructions} by restricting the isomonodromic foliation 
to the locus $s=t$. We postpone the careful study of this picture to another paper. 

\subsection{}\label{SecUncompleteGenus2}

Consider now the family of genus $2$ curves given by 
$$(s,t)\mapsto X_{s,t}=\{y^2=x(x-1)(x-s)(x-t_1)(x-t_2)\},\ \ \ s\in\mathbb C,\ t=(t_1,t_2)\in\mathbb C^2$$
together with the hyperelliptic cover  (see figure \ref{dessin5})
$$\phi_{s,t}:X_{s,t}\to\mathbb P^1_x\ ;\ (x,y)\mapsto x.$$
Let $t\mapsto(E_t,\nabla_t)$ be an isomonodromic deformation of a rank $2$ connection on $\mathbb P^1_x$ with poles located at five among the six critical values, namely $x=0,1,t_1,t_2,\infty$.
Assume that all exponents of $\nabla_t$ take the form $\theta_0=\theta_1=\theta_{t_1}=\theta_{t_2}=\theta_\infty=\frac{1}{2}$.

\begin{figure}[htbp]
\begin{center}
\includegraphics[scale=0.8]{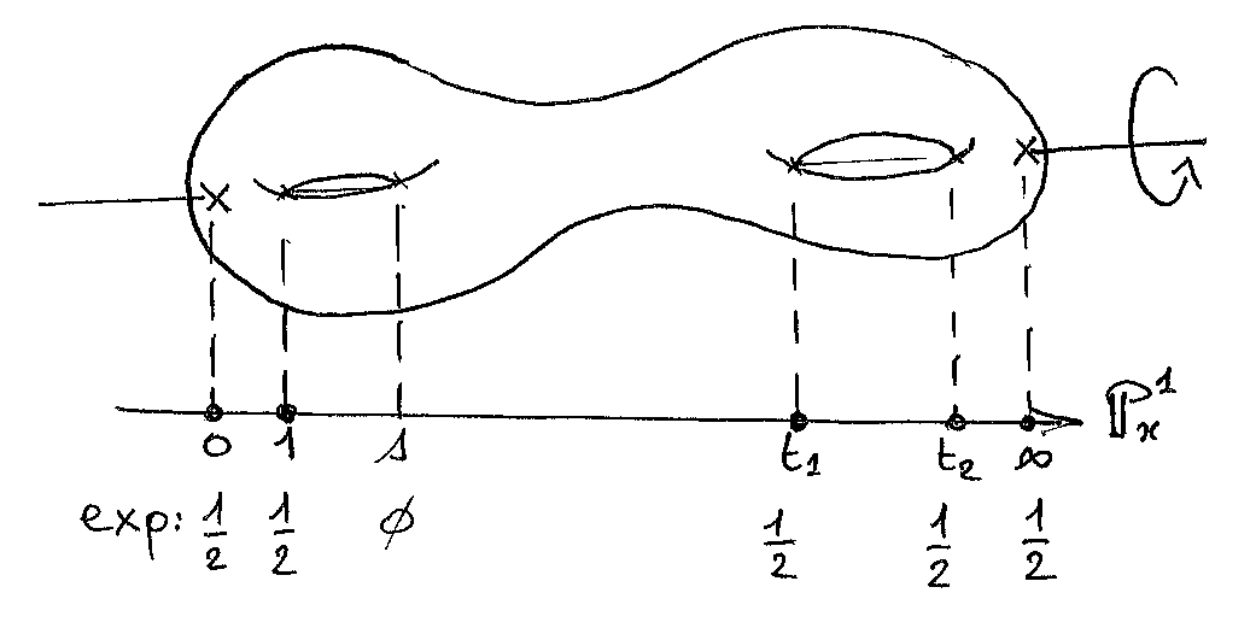}
\caption{{\bf Ruled deformations via uncomplete hyperelliptic cover}}
\label{dessin5}
\end{center}
\end{figure}

After lifting the connection to the curve $X_{s,t}$, deleting apparent singular points by gauge transformation, 
we get a $\sl$-connection on $X_{s,t}$ with a single apparent singular point located at $x=\infty$.
This provides a rational map 
$$\mathbb P^1_s\times\mathcal M_{0,5}(\frac{1}{2},\ldots,\frac{1}{2})\ \longrightarrow\ \mathcal M_{2,1}(1)$$
conjugating isomonodromic foliations. Here, the only singular point is apparent and it is not possible to delete it.
We can just choose to place it at $x=\infty$; it is irrelevant for the deformation.
Again, isomonodromic deformations obtained by this way are parametrized 
by rank $2$ Garnier solutions ($(t_1,t_2)$ variables) combined with a rational function of $s$.
Again, the corresponding leaves of the isomonodromic foliation are uniruled and form a codimension $2$ set.

\subsection{}\label{SecBielliptic1}

Finally, consider the Legendre family $t_1\mapsto X_{t_1}=\{y^2=x(x-1)(x-t_1)\}$ of elliptic curves and let $t=(t_1,t_2)\mapsto(E_t,\nabla_t)$
an isomonodromic deformation of a rank $2$ connection with poles located at $x=0,1,t_1,t_2,\infty$. 
Assume that exponents of $\nabla_t$ take the form $(\theta_0,\theta_1,\theta_{t_1},\theta_{t_2},\theta_\infty)=(\frac{1}{2},\frac{1}{2},\frac{1}{2},\theta,\frac{1}{2})$. After lifting and applying gauge transformation, we get a $\sl$-connection on the elliptic curve $X_{t_1}$ 
with two simple poles over $x=t_2$ having same exponent $\theta$. This gives us a rational map
$$\Phi_\theta\ :\ \mathcal M_{0,5}(\frac{1}{2},\frac{1}{2},\frac{1}{2},\theta,\frac{1}{2})\ \longrightarrow\ \mathcal M_{1,2}(\theta,\theta)$$
conjugating isomonodromic foliations  (see figure \ref{dessin6}). We study this map from the topological (i.e. monodromy) point of view 
in section \ref{SecTopTwoPuncTorus} and deduce 

\begin{thm}\label{ThmBirational}
The map $\Phi_\theta$ is dominant and generically two-to-one.
\end{thm}

In other word, almost all rank $2$ logarithmic connections with two poles on an elliptic curve 
is a pull-back of a rank $2$ logarithmic connection on $\mathbb P^1$; in particular, such connections are 
invariant (up to gauge equivalence) under the hyperelliptic involution permuting the two poles.
This construction can be thought as intermediate between the genus two case and the Lam\'e case of section \ref{SecKnownConstructions}.
This is a reminiscence of the hyperelliptic nature of the twice-punctured torus. 

\begin{figure}[htbp]
\begin{center}
\includegraphics[scale=0.8]{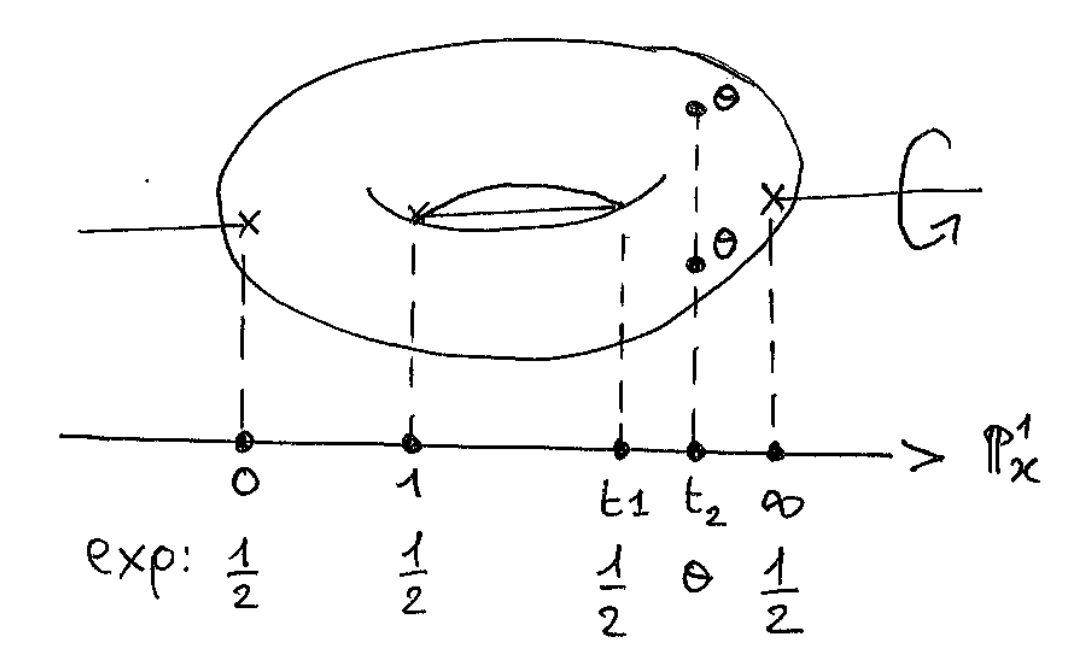}
\caption{{\bf The two punctured torus}}
\label{dessin6}
\end{center}
\end{figure}

\subsection{Classification}

We prove in section \ref{SecClassification} the following

\begin{thm}
Let $t\mapsto(X_t,D_t,E_t,\nabla_t)$ be an isomonodromic deformation of logarithmic $\sl$-connections.
Let $\phi_t:\tilde X_t\to X_t$ a family of ramified covers.
Assume that the pull-back deformation $t\mapsto(\tilde X_t,\tilde D_t,\tilde E_t,\tilde\nabla_t)$
after deleting apparent singular points is locally universal, i.e. the corresponding map $t\mapsto (\tilde X_t,\tilde D_t)$
is locally surjective. In particular, the deformation has dimension $\ge3\cdot\mathrm{genus}(\tilde X_t)-3+\deg(\tilde D_t)$.
Then we are in one of the following cases:
\begin{itemize}
\item The monodromy of $\nabla_t$ (or equivalently $\tilde\nabla_t$) is finite, reducible or dihedral.
\item The deformation $t\mapsto(X_t,D_t,E_t,\nabla_t)$ is actually trivial, and we get an algebraic isomonodromic deformation
by deforming $\phi_t$.
Up to gauge transformation, we are in the list of Doran \cite{Doran} or Diarra \cite{DiarraBoletim}.
In particular, $(X_t,D_t,E_t,\nabla_t)$ is a rigid hypergeometric system ($X_t=\mathbb P^1$, $\deg(D_t)=3$) and $\deg(\phi_t)\le18$.
\item The deformation $t\mapsto(X_t,D_t,E_t,\nabla_t)$ is non trivial, $X_t=\mathbb P^1$, $\deg(\phi_t)=2$ or $4$, 
and we are in one of the constructions described in sections 
\ref{SecQuadratic}, \ref{SecQuartic}, \ref{SecLame}, \ref{SecGenus2}, 
\ref{SecUncompleteTwicePuncturedTorus}, \ref{SecUncompleteGenus2} and \ref{SecBielliptic1}.
\end{itemize}
\end{thm}

\subsection{Complement}\label{SecBielliptic2}

In the last section, we complete the picture of section \ref{SecBielliptic1} when $\theta=\frac{1}{2}$
by constructing a rational map
$$\Psi\ :\ \mathcal M_{1,2}(\frac{1}{2},\frac{1}{2})\ \longrightarrow\ \mathcal M_{2,0}$$
that conjugates isomonodromic foliations.  
In order to explain, consider the ``bi-elliptic cover'' 
$$\xymatrix{
    \tilde X_{t_1,t_2} \ar[r]^{\pi_2} \ar[d]_{\pi_1} \ar[rd]^{\phi} & X_{t_2} \ar[d]^{\phi_2} \\
    X_{t_1} \ar[r]_{\phi_1} & \mathbb P^1_x
}$$
where $\phi_i:X_i\to\mathbb P^1_x$ is the elliptic two-fold cover branching over $x=0,1,t_i,\infty$, for $i=1,2$,
and the remaining part of the diagramm is the fiber product of $\phi_1$ and $\phi_2$. In particular,
$\tilde X_{t_1,t_2}$ has genus $2$ and each $\pi_i: \tilde X_{t_1,t_2}\to X_i$ is a two-fold cover
branching over the two points $\phi_i^{-1}(t_j)$ (where $\{i,j\}=\{1,2\}$). By the way, $\phi:\tilde X_{t_1,t_2}\to\mathbb P^1_x$
is a $4$-fold cover ramifying over all five points $x=0,1,t_1,t_2,\infty$.

\begin{figure}[htbp]
\begin{center}
\includegraphics[scale=0.8]{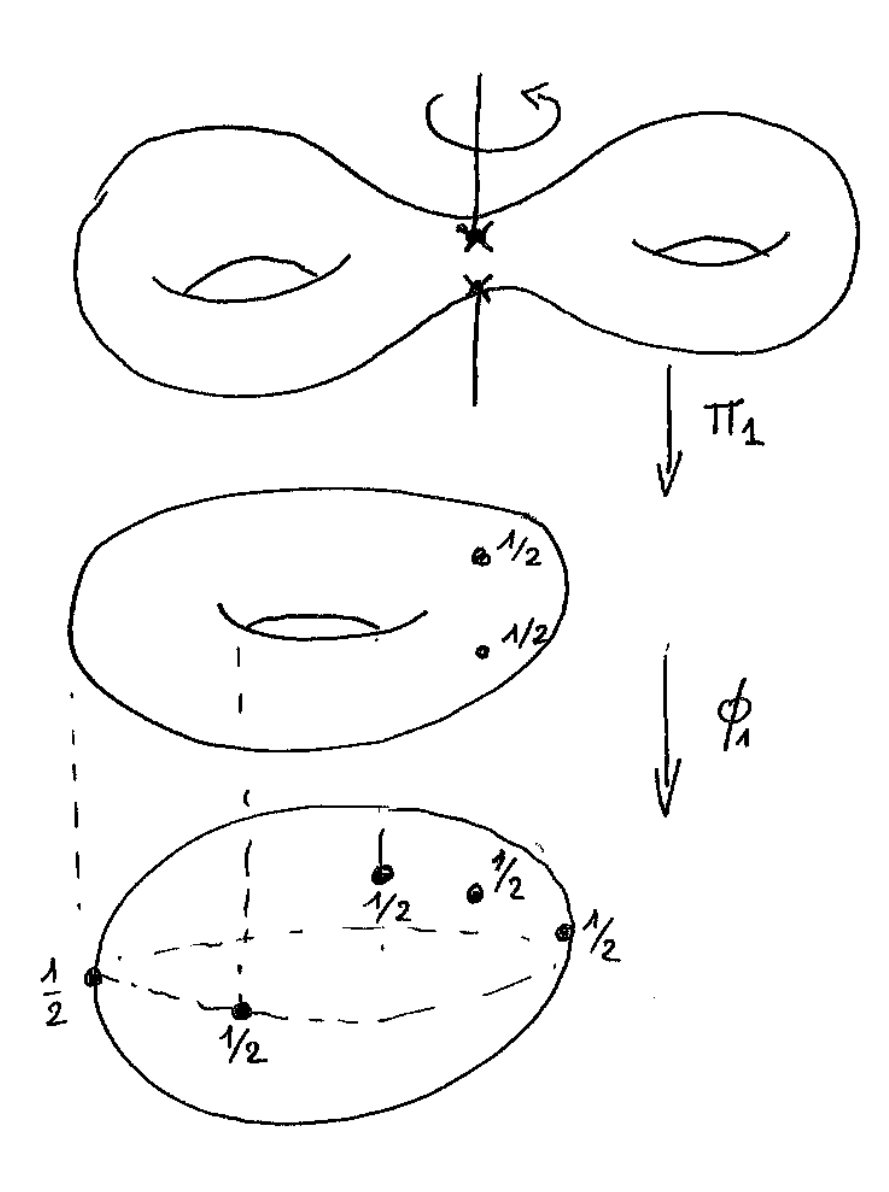}
\caption{{\bf Bi-elliptic cover}}
\label{dessin7}
\end{center}
\end{figure}

The map $\Phi_{\theta}$ of section \ref{SecBielliptic1} comes from the elliptic covering $\pi_1$, while
the map $\Psi$ above, from $\phi_1$ in the bi-elliptic diagramm. In Theorem \ref{ThBielliptic},
we characterize the image of $\Psi$ and 
$$\Psi\circ\Phi_{\frac{1}{2}}\ :\ \mathcal M_{0,5}(\frac{1}{2},\ldots,\frac{1}{2})\ \longrightarrow\ \mathcal M_{2,0}$$
in terms of the monodromy representation. Mind that, contrary to the previous constructions, we do not get complete
isomonodromic deformations (of holomorphic $\sl$-connections on genus $2$ curves) but isomonodromic 
deformations over the codimension $1$ bi-elliptic locus in the moduli space $M_{2,0}$.

This last construction was inspired by \cite{Machu}, where isomonodromic deformations of dihedral logarithmic $\sl$-connections 
are constructed in $\mathcal M_{1,2}(\frac{1}{2},\frac{1}{2})$ as direct image of rank $1$ holomorphic connections on the bi-elliptic 
cover $X_{t_1,t_2}$.

\section{Classification of covers}\label{SecClassification}

Here, we follow ideas of \cite{DiarraBoletim,DiarraToulouse}, replacing connections by their underlying
orbifold structure {\it \`a la Poincar\'e}. 

Let $\phi:\tilde X\to X$ be a ramified cover where $X$ is a genus $g$ hyperbolic orbifold
with $n$ singularities of order $2\le \nu_1\le\cdots\le\nu_n\le\infty$ (i.e. having angle $\alpha_i=\frac{2\pi}{\nu_i}$). 
Pulling-back by $\phi$, we get a {\bf branched} orbifold structure on $\tilde X$: 
orbifold points have angle $\tilde\alpha=\frac{2\pi k}{\nu}$ where $k$ is the branching order of $\phi$ 
(i.e. $\phi\sim z^k$) and 
\begin{itemize}
\item $\nu=\nu_i$ over $i^{\text th}$ orbifold point of $X$, 
\item $\nu=1$ over a regular point. 
\end{itemize}
Denote by $\tilde g$ the genus of $\tilde X$, and by $b$ the number of branching points on $\tilde X$.

The volume of $X$ with respect to the orbifold metric is given by
$$\mathrm{aire}(X)=2\pi(2g-2)+\sum_{i=1}{n}(2\pi-\alpha_i);$$
we get the analogous formula for $\tilde X$ with respect to the pull-back metric
(even if $\alpha_i$ need not be $<2\pi$)
and $\mathrm{aire}(\tilde X)=d\cdot \mathrm{aire}(X)$ where $d=\deg(\phi)$.
This yields (after division by $2\pi$)
\begin{equation}\label{equa:INEGALITEaires}
d\cdot\left(2g-2+\sum_{i=1}^n\left(1-\frac{1}{\nu_i}\right)\right)\ \le\ 2\tilde g-2+\sum_{j=1}^{\tilde n}\left(1-\frac{k_j}{\nu_{i(j)}}\right)-b
\end{equation}
If branching points are simple (with branching order $2$) then we get an equality.

We want to classify cases for which, by deforming simultaneously $X$ and $\phi$,
we get the local universal deformation of $\tilde X$. The dimension of the deformation space of $X$ is given by $3g-3+n\ge0$ 
(positivity ollows from hyperbolicity). For $\tilde X$, since we are more involved in the differential equation than in the orbifold structure, 
we do not take into account the branching points in the deformation, and dimension is given by $3\tilde g-3+\tilde n$.
The dimension of deformation of the ramified cover $\phi$ is given by the number of ``free'' critical values (outside orbifold points)
and thus bounded by $b$. We thus want
 \begin{equation}\label{equa:INEGALITEbranchements}
3g-3+n+b\ge 3\tilde g-3+\tilde n
\end{equation}
On the other hand, it is reasonable to ask
\begin{equation}\label{equa:INEGALITEdimensions}
0<3g-3+n\le 3\tilde g-3+\tilde n
\end{equation}
first because inequality $3g-3+n=0$ corresponds (in the hyperbolic case) to hypergeometric $(g,n)=(0,3)$ 
that has been treated in \cite{DiarraBoletim,DiarraToulouse}; 
right inequality just tells us that we are looking for reductions of isomonodromic equations. 
Throughout the paper, we will also ask $d\ge2$ not to deal with trivial covers. 

Let us first roughly reduce (\ref{equa:INEGALITEaires}) combined
with (\ref{equa:INEGALITEbranchements}). In view of this, let us denote $\nu=\nu_n$ the maximum orbifold 
order (that might be infinite). Then
$$\sum_{i=1}^n\left(1-\frac{1}{\nu_i}\right)\ge \frac{n-1}{2}+\left(1-\frac{1}{\nu}\right).$$
By the same way, we have
$$\sum_{j=1}^{\tilde n}\left(1-\frac{k_j}{\nu_{i(j)}}\right)\le \tilde n \left(1-\frac{1}{\nu}\right).$$
We thus get
\begin{equation}\label{equa:INEGALITEgrossiere}
(2d-3)g+\frac{d-2}{2}n+\tilde g+\frac{\tilde n}{\nu}\ \le\ d\left(\frac{3}{2}+\frac{1}{\nu}\right)-2.
\end{equation}
In fact, we have implicitely assumed $n\not=0$. In the case $n=0$, we automatically get
$\tilde n=0$ and inequality becomes
$$(2d-3)g+\tilde g \le 2d-2;$$
however, we must have $2\le g\le \tilde g$ (hyperbolicity and growth of genus by ramified covers)
that gives us $(2d-2)g\le 2d-2$, contradiction.

\subsection{First bounds}
From the classical Riemann-Hurwitz formula, we necessarily get $\tilde g\ge g$.
After (\ref{equa:INEGALITEgrossiere}), we thus get
$$(2d-2)g\le d\left(\frac{3}{2}+\frac{1}{\nu}\right)-2\le 2d-2.$$
Therefore, we promptly deduce $g\le1$. But when $g=1$, the rough inequality
(\ref{equa:INEGALITEgrossiere}) must be an equality, yielding
$g=\tilde g=1$ and thus (still following Riemann-Hurwitz formula)
$n=\tilde n=0$ and $b=0$. This case is however non hyperbolic.
{\it We can therefore assume $g=0$ from now on.} In particular, $n\ge4$ from (\ref{equa:INEGALITEdimensions}),
and in case $n=4$, hyperbolicity implies $\nu\ge3$.

We can also assume that either $\nu\le d$, or $\nu=\infty$. Indeed, as soon as $\nu>d$,
all points of the fiber $\phi^{-1}(p_n)$ are orbifold; we can therefore
modify the orbifold structure of $X$, replacing $\nu$ by $\infty$, without
modifying the numbers $n$ and $\tilde n$ of orbifold points, and thus without changing dimensions involved in our problem. 

{\bf Assume} $\nu=\infty$. Then (\ref{equa:INEGALITEgrossiere}) gives
$$\frac{d-2}{2}n+\tilde g\ \le\ \frac{3d}{2}-2$$
and thus
$$d\le 2\frac{n-2-\tilde g}{n-3}\le2\frac{n-2}{n-3}.$$
Since $d\ge 2$, we promptly deduce $\tilde g\le 1$, and more precisely,
we are in one of the following cases
\begin{itemize}
\item  $d=2$, $\tilde g\le1$ and $n$ arbitrary,
\item $d=3$, $\tilde g=0$ and $n=4$ or $5$,
\item $d=4$, $\tilde g=0$ and $n=4$.
\end{itemize}
In particular, we get $d\le4$ in this case.

{\bf Assume} $\nu=2$; in this case, $n\ge 5$ because of hyperbolicity.
Then (\ref{equa:INEGALITEgrossiere}) gives
$$d\left(\frac{n}{2}-2\right)\le n-2-\tilde g-\frac{\tilde n}{2}\le n-2-\tilde g-\frac{\tilde n}{3}\le \frac{2n}{3}-2$$
where right inequality follows from (\ref{equa:INEGALITEdimensions}) $3\tilde g+\tilde n\ge n$.
This gives us
$$d\le \frac{4}{3} \frac{n-3}{n-4} <3$$
(because $n\ge 5$) and therefore $d=2$. Taking into account (\ref{equa:INEGALITEgrossiere}),
we get
$$\tilde g+\frac{\tilde n}{2}\le 2.$$
This gives us the following possibilities
\begin{itemize}
\item  $\tilde g=2$ and $\tilde n=0$,
\item  $\tilde g=1$ and $\tilde n\le 2$,
\item  $\tilde g=0$ and $\tilde n\le 4$.
\end{itemize}

{\bf Assume finally} $3\le \nu \le d$.
Then (\ref{equa:INEGALITEgrossiere}) yields
$$d\left(\frac{n}{2}-2+\frac{1}{2}-\frac{1}{\nu}\right)\le n-2-\tilde g-\frac{\tilde n}{\nu}\le n-2-\frac{n}{\nu}-\frac{\nu-3}{\nu}\tilde g$$
where right inequality again follows from (\ref{equa:INEGALITEdimensions}) $3\tilde g+\tilde n\ge n$.
We deduce
$$d\le 2\frac{(n-2)\nu-n}{(n-3)\nu-2}.$$
For each $n>4$, right-hand-side is an increasing function of $\nu$ with asymptotic $2\frac{n-2}{n-3}\le\frac{3}{2}$
when $\nu\to\infty$. Since $\nu<\infty$ here, we get $d<3$ and thus $d=2$; by the way, $\nu\le d\le 2$ and this case is empty.
For  $n=4$, right-hand-side is $4$ whatever is the value of $\nu$. Taking into account (\ref{equa:INEGALITEgrossiere})
for $n=4$ and $d=3,4$, we get
\begin{itemize}
\item  $\tilde g=1$, $\tilde n=1$ (and $\nu=3$),
\item  $\tilde g=0$ and $\tilde n= 4$.
\end{itemize}

\subsection{Degree $d=2$} Here, $\phi$ branches over $2\tilde g+2$ points; recall that $\tilde g\le 2$.
At any orbifold point $p_i$, except when $\nu_i=2$ and $\phi$ branches over $p_i$, 
we can assume $\nu_i=\infty$. In other words, we have say 
\begin{itemize}
\item $n_1$ points with $\nu_i=2$ over which $\phi$ branches, 
\item $n_2=n-n_1$ points with $\nu_i=\infty$ (over which $\phi$ needs not branching).  
\end{itemize}

In the case $\nu=2$, i.e. $n=n_1$ and $n_2=0$, we have already seen that $\tilde g\le 2$,
and thus $n\le 2g+2\le 6$. By hyperbolicity, we must have $n\ge 5$ and we get only two possibilities:
$X$ is an orbifold with $5$ or $6$ conical points $\nu_i=2$ and $\phi:\tilde X\to X$
is a genus $\tilde g=2$ branching over all conical points. {\bf We get examples of sections \ref{SecGenus2} and \ref{SecUncompleteGenus2} respectively}.

Let us now assume $n_2\not=0$ and thus $\nu=\infty$. Coming back to
(\ref{equa:INEGALITEaires}) more carefuly, together with (\ref{equa:INEGALITEbranchements}), we get
$$n_1+2n_2+\tilde g\le 2+n$$
but since $n=n_1+n_2$, we finally get
$$n_2+\tilde g\le 2.$$
Using hyperbolicity assumption (and $n\ge3$), we find the following solutions.
\begin{itemize}
\item  $\tilde g=1$, $n_2=1$ and $3\le n_1\le 4$, 
\item  $\tilde g=0$, $n_2=2$ and $n_1=2$.
\end{itemize}
In the first case, we decompose
\begin{itemize}
\item  $n_1=4$, $\phi$ branches precisely over these $4$ points and $\tilde n=2$, 
\item  $n_1=3$, $\phi$ branches over these $3$ points and one free, and $\tilde n=2$, 
\item  $n_1=3$, $\phi$ branches over $4$ orbifold points and $\tilde n=1$. 
\end{itemize}
{\bf We respectively get examples of sections \ref{SecBielliptic1}, \ref{SecUncompleteTwicePuncturedTorus} and \ref{SecLame}.}
In the second case, $\phi$ branches over the two orbifold points of order $2$ and $\tilde n=4$ {\bf and we get example of section \ref{SecQuadratic}}.

\subsection{Degree $d=3$} We can assume orbifold points of $3$ types:
\begin{itemize}
\item  $\nu_i=2$ and $\phi$ branches at the order $2$ over this point; therefore, the preimage 
consists in one regular point (critical for $\phi$) and a copy of the orbifold point.
\item  $\nu_i=3$ and $\phi$ branches at order $3$ over this point; therefore, the preimage
consists in one regular point (critical for $\phi$).
\item  $\nu_i=\infty$ and $\phi$ is arbitrary over this point; the preimage 
consists in $1$, $2$ or $3$ copies of this point.
\end{itemize}
Denote by $n_2$, $n_3$ and $n_\infty$ the number of these points respectively, $n_2+n_3+n_\infty=n$.
A combination of (\ref{equa:INEGALITEaires}) together with (\ref{equa:INEGALITEbranchements})
yields (with above notations)
$$\tilde g+n+n_\infty=\tilde g+n_2+n_3+2n_\infty\le4$$
This gives us $n=4$ and $\tilde g=n_\infty=0$. But in this case, the only orbifold points
up-stairs have order $2$ and there are at most $4$ such points. This contradict hyperbolicity assumption.

\subsection{Degree $d=4$} We can assume orbifold orders of $4$ types:
\begin{itemize}
\item  $\nu_i=2$ and $\phi$ branches at least once at order $2$ over this point; then the preimage 
consiste consists in one regular point (critical for $\phi$) and either a second one, or two copies of the orbifold point.
\item  $\nu_i=3$ and $\phi$ branches atb order $3$ over this point; then the preimage 
consiste consists in one regular point (critical for $\phi$) and a  copy of the orbifold point.
\item  $\nu_i=4$ and $\phi$ branches at order $4$ over this point; then the preimage 
consiste consists in one regular point (critical for $\phi$).
\item  $\nu_i=\infty$ and $\phi$ is arbitrary over this point; therefore, the preimage
consists in  $1$, $2$, $3$ or $4$ copies of this point.
\end{itemize}
Denote by $n_2$, $n_3$, $n_4$ et $n_\infty$ the number of these points respectively, $n_2+n_3+n_4+n_\infty=n$.
A combination of (\ref{equa:INEGALITEaires}) together with (\ref{equa:INEGALITEbranchements})
yields (with above notations)
$$\tilde g+n_2+2n_3+2n_4+3n_\infty+\frac{\tilde n_2}{2}\le6$$
(here, $\tilde n_2$ is the number of orbifold points of $\tilde X$ over the $n_2$ points of order $2$).
By hyperbolicity, we get $n\ge4$ and, when $n=4$, at least one of the orbifold points is not of minimal order $2$,
yielding $n+n_2+n_3+n_4\ge5$. 

Assume first $n_\infty\not=0$; then, inequalities
allow the only possibility $n=4$ with $(n_2,n_3,n_4,n_\infty)=(3,0,0,1)$, $\tilde g=0$ and $\tilde n_2=0$.
We get the quartic transformation for Painlev\'e VI {\bf (see section \ref{SecQuartic})}.

Let us now assume $n_\infty=0$. Recall that we want $3\tilde g-3+\tilde n=3\tilde g-3+\tilde n_2+n_3\ge1$
if $n=4$ and $\ge2$ if $n\ge5$. From these inequalities, the only possibility is $n=4$ 
with $(n_2,n_3,n_4,n_\infty)=(3,1,0,0)$, $\tilde g=1$ and $\tilde n_2=0$. 
The covering $\phi$ branches only over the $4$ orbifold points, is totally ramified at the order $2$ over the $3$
points of order $2$ and has a single order $3$ branching point over the point of orbifold order $3$. Its monodromy, taking values into the symmetric group $\Sigma_4$, is generated by $3$ double-transpositions
$(ij)(kl)$, $\{i,j,k,l\}=\{1,2,3,4\}$, whose composition has order $3$. However, in $\Sigma_4$, double-transpositions
form a group (together with the identity) and cannot generate an order $3$ element: such a cover does not exist.

\section{From the five-punctured sphere to the twice-punctured torus}\label{SecTopTwoPuncTorus}

Fix distinct points $0,1,t,\lambda,\infty\in\mathbb P^1$, and consider the elliptic cover 
$$\phi:X_{\lambda}:=\{y^2=x(x-1)(x-\lambda)\}\to\mathbb P^1_x\ ;\ (x,y)\mapsto x;$$
denote by $\{t_1,t_2\}:=\phi^{-1}(t)$ the preimage of the fifth point (mind that we change notations).
The orbifold fundamental group of $\mathbb P^1\setminus\{0,1,t,\lambda,\infty\}$ is defined by
$$\Gamma:=\left\langle\gamma_0,\gamma_1,\gamma_t,\gamma_\lambda,\gamma_\infty\ \vert\ 
\gamma_0\gamma_1\gamma_t\gamma_\lambda\gamma_\infty=\gamma_0^2=\gamma_1^2=\gamma_\lambda^2=\gamma_\infty^2=1\right\rangle.$$
On the other hand, the fundamental group of the twice punctured torus $X_{\lambda}\setminus\{t_1,t_2\}$ 
is given by
$$\tilde\Gamma:= \left\langle \alpha, \beta, \delta_1, \delta_2\ \vert\  \alpha\beta=\delta_1\beta\alpha\delta_2\right\rangle.$$ 
The elliptic cover induces a natural monomorphism 
$$\phi_*\ :\ \tilde\Gamma\to\Gamma$$
identifying $\tilde\Gamma$ with an index two subgroup of $\Gamma$: the subgroup generated by $\gamma_t$ 
and words of even length in letters $\gamma_0,\gamma_1,\gamma_\lambda,\gamma_\infty$. In fact,
a careful study of the topological cover yields

\begin{lemma}\label{L1}
The morphism $\phi_*$ is defined by 
$$\left\{ \begin{array}{lll}
\phi_*(\alpha) &= &\tilde\gamma_1\cdot \tilde\gamma_t\cdot\tilde\gamma_\lambda\\
\phi_*(\beta)& =& \tilde\gamma_\lambda\cdot \tilde\gamma_\infty\\
\phi_*(\delta_{1})& =& \tilde\gamma_{t}\\
\phi_*(\delta_{2}) &= &\tilde\gamma_{\infty}\cdot \gamma_{t}\cdot \gamma^{-1}_{\infty}
\end{array}\right.$$
\end{lemma}

One easily check the compatibility between relations defining $\Gamma$ and $\tilde\Gamma$.

\begin{proof} If $p\in\mathbb P^1\setminus\{0,1,t,\lambda,\infty\}$ denotes the base point used to compute
the fundamental group on the sphere, denote by $\tilde p$ and $\tilde p'$ the two lifts on the elliptic curve.
For $i=0,1,\lambda,\infty$, the loop $\gamma_i$ lifts as paths (half loops)
\begin{itemize}
\item $\tilde\gamma_i$ from $\tilde p$ to $\tilde p'$,
\item $\tilde\gamma_i'$ from $\tilde p'$ to $\tilde p$.
\end{itemize}
On the other hand, the loop $\gamma_t$ lifts as loops
\begin{itemize}
\item $\tilde\gamma_t$ based at $\tilde p$,
\item $\tilde\gamma_t'$based at $\tilde p'$.
\end{itemize}
Then, carefully drawing the picture, we get 
$$\left\{ \begin{array}{lll}
\alpha &= &\tilde\gamma_1\cdot \tilde\gamma'_t\cdot\tilde\gamma'_\lambda\\
\beta& =& \tilde\gamma_\lambda\cdot \tilde\gamma'_\infty\\
\delta_{1}& =& \tilde\gamma_{t}\\
\delta_{2} &= &\tilde\gamma_{\infty}\cdot \gamma'_{t}\cdot \gamma^{-1}_{\infty}
\end{array}\right.$$
We check that these loops indeed satisfy $\alpha\beta=\delta_1\beta\alpha\delta_2$
by using relations
$$\tilde\gamma_i\cdot\tilde\gamma'_i=1\ \ \ \text{for}\ \ \ i=0,1,\lambda,\infty$$
and those which lift as $\gamma_0\circ\gamma_1\circ\gamma_t\circ\gamma_{\lambda}\circ\gamma_{\infty}=1$
namely
$$\tilde\gamma_0\circ\tilde\gamma'_1\circ\tilde\gamma_t\circ\tilde\gamma_{\lambda}\circ\tilde\gamma'_{\infty}=1
\ \ \ \text{and}\ \ \ 
\tilde\gamma'_0\circ\tilde\gamma_1\circ\tilde\gamma'_t\circ\tilde\gamma'_{\lambda}\circ\tilde\gamma_{\infty}=1.$$
We get the result by projection on $\mathbb P^1_x$.
\end{proof}

 \begin{lemma}\label{L2} 
 The unique elliptic involution of $X_{t_1}$ that permutes  $t_1$ and $t_2$ acts as follows on the fundamental group:
 $$ \begin{array}{ccccc}\alpha&\leftrightarrow&\alpha^{-1}\\
 \beta&\leftrightarrow&\beta^{-1}\\
\gamma_1&\leftrightarrow&\gamma_2\end{array}$$
\end{lemma}

We note that the relation $\alpha\beta=\delta_1\beta\alpha\delta_2$ is indeed invariant by the involution.

\begin{proof}We have to take care that the base point $\tilde p$ is not fixed.
In fact, the involution permutes $\tilde p$ and $\tilde p'$ and acts on  $\gamma_i$ lifts as follows 
$$\tilde\gamma_i\leftrightarrow\tilde\gamma'_i\ \ \ \text{for}\ \ \ i=0,1,t,\lambda,\infty.$$
In particular, if we denote
$$\left\{ \begin{array}{lll}
\alpha' &= &\tilde\gamma'_1\cdot \tilde\gamma_t\cdot\tilde\gamma_\lambda\\
\beta'& =& \tilde\gamma'_\lambda\cdot \tilde\gamma_\infty
\end{array}\right.$$ 
then involution acts on these loops as 
$$\alpha\leftrightarrow\alpha'\ \ \ \text{and}\ \ \  \beta\leftrightarrow\beta'.$$
We bring back these new loops to the base point $\tilde p$ by conjugating (for instance) with $\tilde\gamma_\infty$, which gives us
$$ \begin{array}{ccccc}\alpha&\leftrightarrow&\tilde\gamma_\infty\cdot\alpha'\cdot\tilde\gamma_\infty^{-1}\\
 \beta&\leftrightarrow&\tilde\gamma_\infty\cdot\beta^{-1}\cdot\tilde\gamma_\infty^{-1}\\
\tilde\gamma_t&\leftrightarrow&\tilde\gamma_\infty\cdot\tilde\gamma'_t\cdot\tilde\gamma_\infty^{-1}\end{array}$$
We thus get $\delta_1\leftrightarrow\delta_2$ and, by a direct computation, using relations between $\tilde\gamma_i$
and $\tilde\gamma'_i$, we check that $\alpha\leftrightarrow\alpha^{-1}$ and  $\beta\leftrightarrow\beta^{-1}$.
\end{proof}

In order to prove Theorem \ref{ThmBirational}, it is enough to prove that the map $\Phi_\theta$ is dominant,  generically two-to-one.
By the Riemann-Hilbert correspondance, it is equivalent to work with the corresponding spaces of monodromy representations.
Let us denote by $\mathcal R_\theta$ the space of monodromy representations for $\mathcal M_{0,5}(\frac{1}{2},\frac{1}{2},\frac{1}{2},\theta,\frac{1}{2})$:
$$\mathcal R_\theta:=\left\{(M_0,M_1,M_t,M_\lambda,M_\infty)\in\SL_2(\C)^5\ ;\ \ \ 
\begin{array}{c}M_0M_1M_tM_\lambda M_\infty=I\\
\mathrm{trace}(M_i)=0\ \text{for}\ i=0,1,\lambda,\infty\\ 
\mathrm{trace}(M_t)=2\cos(\pi\theta)
\end{array}\right\}/\sim$$
where the equivalence relation $\sim$ is the diagonal adjoint action by $\SL_2(\C)$ on quintuples.
Recall that, in $\SL_2(\C)$, we have 
$$\mathrm{trace}(M)=0\Leftrightarrow M^2=-I$$
and the corresponding $\mathrm{PSL}_2(\C)$-representations are actually representations
$$\Gamma\to\mathrm{PSL}_2(\C).$$
On the other hand, consider the space $\tilde{\mathcal R}_{\theta}$ of monodromy representations of $\mathcal M_{1,2}(\theta,\theta)$
$$\tilde{\mathcal R}_\theta:=\left\{(A,B,D_1,D_2)\in\SL_2(\C)^4\ ;\ \ \ 
\begin{array}{c}AB=D_1 BA D_2\\
\mathrm{trace}(D_1)=\mathrm{trace}(D_2)=2\cos(\pi\theta)
\end{array}\right\}/\sim$$
The natural map ${\phi_1}^*:\mathcal R_\theta\to\tilde{\mathcal R_\theta}$ induced by ${\phi_1}$ is described by

\begin{cor}\label{CorMA}
We have ${\phi_1}^*(M_0,M_1,M_t,M_\lambda,M_\infty)=(A,B,D_1,D_2)$ with
$$\left\{ \begin{array}{ccc}
 A& =& M_1 M_t M_{\lambda},\\
 B &=& M_\lambda M_\infty,\\
 D_1& =& M_{t},\\
 D_2& = &M_\infty M_t M^{-1}_\infty.\end{array} \right.$$
\end{cor}

\begin{proof}
From Lemma \ref{L1}, we know that $AB=\pm D_1 B A D_2$; we just have to check that we have the right sign,
and thus a representation
$$\pi_1(X_\lambda\setminus\{t_1,t_2\})\to\SL_2(\C)$$
and we must have 
$\mathrm{trace}(D_1)=\mathrm{trace}(D_2)=2\cos(\pi\theta)\ \ (=\mathrm{trace}(M_t))$.
\end{proof}

We now want to prove that the map  ${\phi_1}^*:\mathcal R_\theta\to\tilde{\mathcal R_\theta}$ just defined is generically one-to-one.
This follows from the following

\begin{thm}\label{R0}
Let $A,B,D_1,D_2\in\SL_2(\C)$ such that
$$AB=D_1 B A D_2\ \ \ \text{and}\ \ \ D_1,D_2\not=\pm I.$$
Assume moreover that the subgroup $<A,B>$ generated by $A$ and $B$ is irreducible, i.e. without common eigendirection. 
Then there is a matrix $M\in\SL_2(\C)$, unique up to a sign, such that
$$MAM^{-1}=A^{-1},\ \ \ MBM^{-1}=B^{-1}\ \ \ \text{and}\ \ \ MD_1M^{-1}=D_2.$$
Moreover, $M^2=-I$ and $(A,B,D_1,D_2)={\phi_1}^*(M_0,M_1,M_t,M_\lambda,M_\infty)$ for
$$\left\{ \begin{array}{ccc}
 M_0&=& -AM\\
 M_1&=&ABD_2^{-1}M\\
 M_t&=&D_1\\
 M_\lambda&=&-BM\\
 M_\infty&=&M
 \end{array}\right.$$
\end{thm}

First recall well-known results concerning $\SL_2(\C)$.

\begin{lemma}\label{LemBasic1} Two matrices $A,B\in\SL_2(\C)$ generate a reducible group if, and only if, 
$\mathrm{trace}[A,B]=2$ where $[A,B]=ABA^{-1}B^{-1}$ is the commutator.
\end{lemma}

\begin{proof}If $A$ et $B$ have a common eigenvector, then we can assume $<A,B>$ is triangular and the commutator
will be a unipotent matrix, thus having trace $2$. Conversely, assume that $A$ and $B$ have no common eigenvector. 
Therefore, an eigenvector $v$ for $AB$ will not be eigenvector for $A$ or for $B$. If $ABv=\gamma v$, then in the base $(v,-\gamma Bv)$, 
matrices take the form
$$A=\left(\begin{array}{cc}
a&-1\\
1&0
\end{array}\right)\ \ \ \text{and}\ \ \ B=\left(\begin{array}{cc}
0&\frac{1}{\gamma}\\
-\gamma&b
\end{array}\right)$$
where $a=\mathrm{trace}(A)$ and $b=\mathrm{trace}(B)$.
We check that
$$[A,B]=\begin{pmatrix}
a^2+b^2+\gamma^2-abc&\gamma^{-2}(a-b\gamma)\\
a-b\gamma^{-1}&\gamma^{-2}\end{pmatrix}$$
and thus 
$$\mathrm{trace}([A,B])=a^2+b^2+c^2-abc-2,\ \ \ c=\gamma+\gamma^{-1}=\mathrm{trace}(AB).$$
Finally, these matrices $A$ and $B$ have a common eigenvector if, and only if, $a^2+b^2+c^2-abc-2=2$.
\end{proof}

\begin{lemma}\label{LemBasic2} Let $A,B,A',B'\in\SL_2(\C)$ and assume $\mathrm{trace}[A,B]\not=2$. There exists $M\in\SL_2(\C)$ such that 
$$MAM^{-1}=A'\ \ \ \text{and}\ \ \ MBM^{-1}=B'$$ if, and only if, 
$$\mathrm{trace}(A)=\mathrm{trace}(A'),\ \ \ \mathrm{trace}(B)=\mathrm{trace}(B')\ \ \ \text{and}\ \ \ \mathrm{trace}(AB)=\mathrm{trace}(A'B').$$
\end{lemma}

\begin{proof}
This is a consequence of formulae from the preceeding proof.
\end{proof}

\begin{cor}\label{CorBasic} If $\mathrm{trace}[A,B]\not=2$, then there exists $M\in\SL_2(\C)$, unique up to a sign, such that 
$$MAM^{-1}=A^{-1}\ \ \ \text{and}\ \ \ MBM^{-1}=B^{-1}.$$
Moreover, $M^2=-I$.
\end{cor}

\begin{proof} It suffices to notice that $\mathrm{trace}(A)=\mathrm{trace}(A^{-1})$ and $\mathrm{trace}(AB)=\mathrm{trace}(BA)$ 
for all matrices $A,B\in\SL_2(\C)$. We deduce, under our assumptions, that
$$\mathrm{trace}(A)=\mathrm{trace}(A^{-1}),\ \ \ \mathrm{trace}(B)=\mathrm{trace}(B^{-1})\ \ \ \text{and}\ \ \ \mathrm{trace}(AB)=\mathrm{trace}(A^{-1}B^{-1}).$$
Therefore, there exists an $M$ satisfying the first part of the statement. But $M^2$ has to commute to $A$ and $B$. 
Thus $M^2$ must fix all eigendirections of all elements of the group $<A,B>$. There are at least three distinct such directions
and  $M^2$ is projectively the identity: $M^2=\pm I$. But $M=\pm I$ is impossible since $MAM^{-1}=A^{-1}\not=A$ ($A\not=\pm I$ otherwise $<A,B>$ would be reductible). Thus $M^2\not=I$ and $M^2=-I$. If matrices $A$ and $B$ are given in the normal form like in the proof above, 
then $M$ is given by
\begin{equation}\label{InvolMatrix}
M=\pm\begin{pmatrix}
\frac{\gamma^2-1}{2\gamma}&\frac{a-b\gamma}{2\gamma}\\ 
\frac{a\gamma-b}{2}&-\frac{\gamma^2-1}{2\gamma}\end{pmatrix}
\end{equation}
\end{proof}

\begin{proof}[Proof of Theorem \ref{R0}]
We want now to prove that the unique (up to a sign) matrix $M$ satisfying
$$MAM^{-1}\ \ \ \text{and}\ \ \ MBM^{-1}$$
also satisfy
$$MD_1M^{-1}=D_2\ \ \ \text{and thus}\ \ \ MD_2M^{-1}=D_1$$
($M^2=-I$). From relation $AB=D_1BAD_2$, this is equivalent to
$$AB=D_1BAMD_1M^{-1}\Leftrightarrow (BAMD_1)^2=-I\Leftrightarrow \mathrm{trace}(BAMD_1)=0.$$
Rewrite the relation $AB=D_1BAD_2$ into the form 
$$[A,B]=D_1BAD_2A^{-1}B^{-1}=D_1D'_2\ \ \ \text{with}\ \ \ D'_2=(BA)D_2(BA)^{-1}.$$
Note that
$$(BAM)^2=BAMBAM=BAB^{-1}A^{-1}M^2=-BAB^{-1}A^{-1}=-[A,B]^{-1}$$
and therefore $(BAM)^2D_1=-(D'_2)^{-1}$ and
$$\mathrm{trace}((BAM)^2D_1)+\mathrm{trace}(D_1)=0.$$
Now, recall that in $\SL_2(\C)$ we have universal relations 
$$\mathrm{trace}(M_1M_2)+\mathrm{trace}(M_1M_2^{-1})=\mathrm{trace}(M_1)\cdot\mathrm{trace}(M_2).$$
Applying this to $M_1=BAM$ and $M_2=BAMD_1$, we get
$$0=\mathrm{trace}((BAM)^2D_1)+\mathrm{trace}(D_1)=\mathrm{trace}(BAMD_1)\cdot\mathrm{trace}(BAM).$$
But, $\mathrm{trace}(BAM)\not=0$ otherwise $(BAM)^2=-[A,B]^{-1}=-I$, i.e. $[A,B]=I$, that would contradict irreducibility. 
Thus $\mathrm{trace}(BAMD_1)=0$, what we wanted to prove. Finally, we easily check that matrices $M_i$ given by the statement
are indeed inversing preceeding formulae of Lemma \ref{CorMA} by using relation $AB=D_1BAD_2$ and properties of $M$.
\end{proof}

\section{Bielliptic covers}\label{sectionMonodromieBielliptique}

Let us now assume $\theta=0$ and rewrite
$$\tilde{\mathcal R}_{1/2}:=\left\{(A,B,C_1,C_2)\in\SL_2(\C)^4\ ;\ \ \ 
\begin{array}{c}[A,B]=C_1C_2\\
\mathrm{trace}(C_1)=\mathrm{trace}(C_2)=0
\end{array}\right\}/\sim$$
where we have modified generators of the fundamental group for convenience:
$$C_1=D_1\ \ \ \text{and}\ \ \ C_2=(BA)^{-1} D_2 (BA).$$
This is the monodromy space of those connections on the elliptic curve $X_\lambda$ having logarithmic poles
with exponent $\frac{1}{2}$ at $t_1$ and $t_2$. Let us now consider the $2$-fold ramified cover
$$\pi:\tilde X_{t,\lambda}\to X_\lambda$$
ramifying over $t_1$ and $t_2$ and let us study the associated map 
$$\pi^*:\mathcal M_{1,2}(\frac{1}{2},\frac{1}{2})\to\mathcal M_{2,0}$$
on the monodromy side of the Riemann-Hilbert correspondance.
Denote by 
$$\mathcal R':=\left\{(A_1,B_1,A_1,B_2)\in\SL_2(\C)^4\ ;\ \ \ [A_1,B_1][A_2,B_2]=I\right\}/\sim$$
the space of monodromy representations associated to $\mathcal M_{2,0}$.
Then we get a map 
$$\pi^*:\tilde{\mathcal R}_{1/2}\longrightarrow \mathcal R'$$
which is given by (see also \cite{Machu})

\begin{lemma}\label{L3}
We have ${\pi}^*(A,B,C_1,C_2)=(A_1,B_1,A_1,B_2)$ with
$$\left\{ \begin{array}{ccc}
 A_1& =& A,\\
 B_1 &=& B,\\
 A_2 & =& C_1^{-1}AC_1,\\
 B_2 & = & C_1^{-1} BC_1.\end{array} \right.$$
% The Galois involution of the cover is just the permutation 
% $$A_1\leftrightarrow A_2\ \ \ \text{and}\ \ \ B_1\leftrightarrow B_2.$$
\end{lemma}

Conversely, we can characterize the image of $\pi^*$ as follows

\begin{thm}\label{ThBielliptic} Let $A_1,B_1,A_2,B_2\in\SL_2(\C)$ such that
$$[A_1,B_1][A_2,B_2]= I.$$
Assume that there exists a matrix $M\in\SL_2(\C)$ such that 
$$MA_1M^{-1}=A_2,\ \ \ MB_1M^{-1}=B_2\ \ \ \text{and}\ \ \ M^2=-I.$$
Then $(A_1,B_1,A_2,B_2)=\pi^*(A,B,C_1,C_2)$ for 
$$\left\{\begin{array}{ccc}
 A& =& A_1,\\ 
 B &=& B_1,\\
  C_1& =& M\\  
  C_2 &= &M^{-1}[A_1,B_1],
  \end{array}\right.$$
If moreover 
$$\mathrm{trace}[A_1,B_1]\not=2$$
then $(A_1,B_1,A_2,B_2)$ is in the image of $\pi\circ\phi$, i.e. comes from a representation
of the $5$-punctured sphere.
\end{thm}

\begin{remark}From Lemma \ref{LemBasic2}, we see that existence of $M$ is almost equivalent to
$$\mathrm{trace}(A_1)=\mathrm{trace}(A_2)=:a\ \ \ \text{and}\ \ \ \mathrm{trace}(B_1)=\mathrm{trace}(B_2)=:b.$$
To apply the Lemma, we just need to prove that the two traces $c_i:=A_iB_i$ coincide for $i=1,2$.
But the relation $[A_1,B_1][A_2,B_2]= I$ implies that the two commutators are inverse to each other,
and thus share the same trace. By the commutator trace formula in the proof of Lemma \ref{LemBasic1}, we get
$$(c_1-c_2)(c_1+c_2-ab)=0.$$
The image of $\pi^*$ has codimension $2$ in $\mathcal R'$. We also see that generic fibers of $\pi^*$ consist in $2$ points.
\end{remark}

\begin{remark}
If we fix $A_1$ and $B_1$ generic, we obtain:
\begin{enumerate}
\item the set $\{M\in \SL_2(\C)\ ;\  \textrm{$M^2=-I $ and $M$ conjugates $[A_1, B_1] $ to its inverse}\}$ has dimension $1$,
\item the set $\{ A_1, B_1, M^{-1}.A_1.M, M^{-1}.B_1.M\}$ has also dimension $1$ up to conjugacy.
\end{enumerate}
Thus we can freely choose $(A_1,B_1)$ in the image of $\pi^*$.
\end{remark}

\end{document}